\documentclass[12pt,oneside,openany,article]{memoir}

\nouppercaseheads

\usepackage[amsthm,thmmarks,hyperref]{ntheorem}

\usepackage[noTeX]{mmap}

\usepackage[english.us]{babel}
\usepackage[leqno]{mathtools}

\usepackage{caption}
\usepackage[scr]{rsfso}
\usepackage{upgreek}

\usepackage[compress]{cite}
\usepackage{hypernat} %required if hyperref is used as well

\usepackage{fix-cm}
\usepackage[cm]{sfmath}
\usepackage{sansmathaccent}

\usepackage{microtype}

\usepackage{subfig}
\usepackage{wrapfig}
\usepackage{array}

\usepackage{graphicx,xcolor}
\usepackage{eso-pic}

\usepackage{pdfpages}
\usepackage{hyperxmp}
\usepackage{zref, nameref}

\usepackage{xr-hyper}
\usepackage{url}
\usepackage[pdftex,bookmarks,pdfnewwindow,plainpages=false,unicode,pdfencoding=auto, backref=page]{hyperref}
\usepackage{backref}

\usepackage{bookmark}

\usepackage{enumitem}

\usepackage{amssymb} % for \Subset, \lozenge, etc.

\hypersetup{
colorlinks=true,
linkcolor=black,
citecolor=black,
urlcolor=blue,
pdfauthor={Victor Ivrii},
pdftitle={Asymptotics of the ground state energy in the relativistic settings},
pdfsubject={Asymptotics of the ground state energy in the relativistic settings},
pdfkeywords={Relativistic Schroedinger operator, Heavy atoms and Molecules, Thomas Fermi theory, Scott correction term, Microlocal Analysis,  Sharp Spectral Asymptotics},
bookmarksdepth={4}
}
\theoremstyle{plain}
\newtheorem{theorem}{Theorem}[chapter]

\newtheorem{proposition}[theorem]{Proposition}
\newtheorem{corollary}[theorem]{Corollary}

\theoremstyle{definition}

\theoremstyle{remark}
\newtheorem{remark}[theorem]{Remark}
\makeatletter
\newtheoremstyle{plainfoot}%
  {\item[\hskip\labelsep \theorem@headerfont ##1\ ##2\,\footnotemark\theorem@separator]}%
  {\item[\hskip\labelsep \theorem@headerfont ##1\ ##2\ (##3)\, \footnotemark\theorem@separator]}
\makeatother

\theoremstyle{plainfoot}
\theoremseparator{.}
\newtheorem{theorem-foot}[theorem]{Theorem}
\newtheorem{lemma-foot}[theorem]{Lemma}
\newtheorem{proposition-foot}[theorem]{Proposition}
\newtheorem{corollary-foot}[theorem]{Corollary}
\newtheorem{conjecture-foot}[theorem]{Conjecture}
\newtheorem{condition-foot}[theorem]{Condition}

\theoremstyle{plainfoot}
\theoremseparator{.}
\theorembodyfont{}
\newtheorem{definition-foot}[theorem]{Definition}
\newtheorem{Problem-foot}[theorem]{Problem}

\theoremstyle{plainfoot}
\theoremseparator{.}
\theorembodyfont{}
\theoremheaderfont{\slshape}
\newtheorem{remark-foot}[theorem]{Remark}         % remark with no number
\newtheorem{example-foot}[theorem]{Example}
\newtheorem{problem-foot}[theorem]{Problem}

\DeclareTextCommand{\textvartheta}{PU}{\83\321}

\numberwithin{equation}{chapter}

\DeclareMathAlphabet{\mathpzc}{OT1}{pzc}{m}{it}

\newcommand{\cE}{\mathcal{E}}

\newcommand{\cX}{\mathcal{X}}
\newcommand{\cY}{\mathcal{Y}}

\newcommand{\sH}{\mathscr{H}}

\newcommand{\sL}{\mathscr{L}}

\newcommand{\scl}{\mathsf{scl}}
\newcommand{\rel}{\mathsf{rel}}
\newcommand{\sing}{{{\operatorname  {sing}}}}

\newcommand{\E}{{\mathsf{E}}}
\newcommand{\TF}{{\mathsf{TF}}}

\newcommand{\D}{{\mathsf{D}}}

\newcommand{\N}{{\mathsf{N}}}

\newcommand{\x}{{\mathsf{x}}}
\newcommand{\y}{{\mathsf{y}}}

\newcommand{\bC}{{\mathbb{C}}}

\newcommand{\bR}{{\mathbb{R}}}

\newcommand{\fH}{{\mathfrak{H}}}

\newcommand{\blangle}{{\boldsymbol{\langle}}}
\newcommand{\brangle}{{\boldsymbol{\rangle}}}

\newcommand{\3}{{|\!|\!|}}

\newcommand{\Spec}{\operatorname{Spec}}

\newcommand{\tr}{\operatorname{tr}}
\newcommand{\Tr}{\operatorname{Tr}}

\newenvironment{claim}[1][{\textup{(\theequation)}}]{\refstepcounter{equation}\vglue10pt
\begin{trivlist}
\item[{\hskip\labelsep#1}]}{\vglue10pt\end{trivlist}}

\newenvironment{claim*}[1][{}]{\vglue10pt
\begin{trivlist}
\item[]}{\vglue10pt\end{trivlist}}

\newenvironment{phantomequation}[1][]{\refstepcounter{equation}}{}
\newcounter{note}

\newcommand{\RTF}{{\mathsf{RTF}}}

\newcommand{\e}{{\mathsf{e}}}

\numberwithin{equation}{chapter}

\title{Asymptotics of the ground state energy in the relativistic settings}

\author{Victor Ivrii}

\begin{document}

\maketitle

%\tableofcontents

\begin{abstract}
The purpose of this paper is to derive sharp asymptotics of  the ground state energy for the heavy atoms and molecules in the relativistic settings, and, in particular, to derive relativistic Scott correction term and also Dirac, Schwinger and relativistic correction terms. Also we will prove that Thomas-Fermi density approximates the actual density of the ground state, which opens the way to estimate the excessive negative and positive charges and the ionization energy.
\end{abstract}

\chapter{Introduction}
\label{sect-1}

The purpose of this paper is to derive sharp asymptotics of  the ground state energy for the heavy atoms and molecules in the relativistic settings, and, in particular, to derive relativistic Scott correction term and also Dirac, Schwinger and relativistic correction terms. The relativistic Scott correction term was first derived in \cite{SSS} which both inspired our paper and provided necessary functional analytic tools; our improvement is achieved due to more refined microlocal semiclassical technique.

Also we will prove that Thomas-Fermi density approximates the actual density of the ground state, which opens the way to estimate the excessive negative and positive charges and the ionization energy.

 In the next article we plan to introduce a self-generated magnetic field and improve results of \cite{EFS2}.

Multielectron Hamiltonian is given by
\begin{gather}
\mathsf{H}=\mathsf{H}_N\coloneqq   \sum_{1\le j\le N} H _{V,x_j}+\sum_{1\le j<k\le N}\frac{\e^2}{|x_j-x_k|}
\label{1-1}\\
\shortintertext{on}
\fH= \bigwedge_{1\le n\le N} \sH, \qquad \sH=\sL^2 (\bR^3, \bC^q)\simeq \sL^2 (\bR^3\times \{1,\ldots,q\},\bC) 
\label{1-2}\\
\shortintertext{with}
H_V =T - \e V(x),
\label{1-3}
\end{gather}
describing $N$ same type particles in the external field with the scalar potential $-V$  and repulsing one another according to the Coulomb law; $\mathsf{e}$ is a charge of the electron, $T$ is an \emph{operator of the kinetic energy}.

In the non-relativistic framework this operator is defined as
\begin{align}
&T= \frac{1}{2\mu} (-i\hbar \nabla)^2.
\label{1-4}
\end{align}
In the relativistic framework this operator is defined as
\begin{align}
&T= \Bigl(c^2 (-i\hbar\nabla)^2 +\mu ^2 c^4\Bigr)^{\frac{1}{2}}-\mu^2c^4 ,
\label{1-5}
\end{align}
in the non-magnetic, magnetic (Schr\"odinger) and magnetic (Schr\"odinger-Pauli) settings respectively.

Here
\begin{gather}
V(x)=\sum _{1\le m\le M}  \frac{Z_m \e}{|x-\y_m|}
\label{1-6}\\
\shortintertext{and}
d=\min _{1\le m<m'\le M}|\y_m-\y_{m'}|>0.
\label{1-7}
\end{gather}
where $Z_m\mathsf{e}>0$ and $\y_m$ are charges and locations of nuclei. 

It is well-known that the non-relativistic operator is always semibounded from below. On the other hand, it 
 is also well-known \cite{Herbst, Lieb-Yau} that
\begin{claim}\label{1-8}
One particle relativistic non-magnetic operator is semibounded from below if and only if
\begin{equation}
Z_m \beta \le \frac{2}{\pi}\qquad \forall m=1,\ldots, M;\qquad \beta\coloneqq \frac{\e^2}{\hbar c}.
\label{1-9}
\end{equation}
\end{claim}
We will assume (\ref{1-9}), sometimes replacing it by a strict inequality:
\begin{equation}
Z_m \beta \le \frac{2}{\pi}-\epsilon\qquad \forall m=1,\ldots, M;\qquad \beta\coloneqq \frac{\e^2}{\hbar c}.
\label{1-10}
\end{equation}
We also assume that $d\ge CZ^{-1}$. Then we are interested in $\E\coloneqq \inf\Spec(\mathsf{H})$.  

\begin{remark}\label{rem-1-1}
\begin{enumerate}[label=(\roman*), wide, labelindent=0pt]
\item\label{rem-1-1-i}
In the non-relativistic theory by scaling with respect to the spatial and energy variables we can make $\hbar=\e=\mu=1$ while $Z_m$ are preserved.

\item\label{rem-1-1-ii}
In the relativistic theory by scaling with respect to the spatial and energy variables we can make $\hbar=\e=\mu=1$ while $\beta$ and $Z_m$ are preserved.
\end{enumerate}
From now on we assume that such rescaling was done and we are free to use letters $\hbar$, $\mu $ and $c$ for other notations.
\end{remark}

\chapter{Functional analytic arguments}
\label{sect-2}

\section{Estimate from below}
\label{sect-2-1}
In contrast to \cite{SSS} we will start from the more traditional approach. We estimate 
$\sum _{1\le j<k \le N}\blangle |x_j-x_k|^{-1}\Psi ,\Psi\brangle $ from below using Lieb's electrostatic inequality by
$\frac{1}{2}\D (\rho_\Psi,\Psi) -C\int \rho_\Psi^{4/3}\,dx$ where where $\blangle \cdot,\cdot\brangle $ means the inner product in $\fH$,  $\rho_\Psi(x)$ is a one particle density, and we use notations of  Chapter~\ref{book_new-sect-25} of \cite{monsterbook}.

The  the standard estimate (\ref{book_new-25-2-2}) from \cite{monsterbook} from below holds:
\begin{multline}
\blangle \mathsf{H}_N\Psi ,\Psi \brangle \ge
\sum_{1\le j\le N} \blangle H_{V,x_j}\Psi ,\Psi \brangle +
\frac{1}{2}\D\bigl(\rho_\Psi ,\rho_\Psi )-
C\int \rho_\Psi ^{\frac{4}{3}}(x)\,dx=\\[3pt]
\sum_{1\le j\le N}
\blangle H_{W,x_j}\Psi ,\Psi \brangle +
\frac{1}{2}\D\bigl(\rho_\Psi - \rho ,\rho_\Psi -\rho \bigr)-
\frac{1}{2}\D\bigl( \rho , \rho \bigr)-C\int \rho_\Psi ^{\frac{4}{3}}(x)\,dx
\label{2-1}%{25-2-2}
\end{multline}
where $H_W$ is one-particle Schr\"odinger (etc) operator with the potential
\begin{equation}
W=V-|x| ^{-1}*\rho ,
\label{2-2}%{25-2-3}
\end{equation}
where $\rho$ is an arbitrary chosen real-valued non-negative function. Then again we get
\begin{equation}
\E_N \ge \Tr (H^-_{W+\lambda})+\lambda N + \frac{1}{2}\D\bigl(\rho_\Psi - \rho ,\rho_\Psi -\rho \bigr)-
\frac{1}{2}\D\bigl( \rho , \rho \bigr)-C\int \rho_\Psi ^{\frac{4}{3}}(x)\,dx
\label{2-3}
\end{equation}
with arbitrary $\lambda$. 

\begin{remark}\label{rem-2-1} 
As usual, we will need to improve these estimates to recover remainder estimate better than 
$O(Z^{\frac{5}{3}})$. 
\end{remark}

Now we need to prove estimate 
\begin{equation}
\int \rho_\Psi ^{\frac{4}{3}}(x)\,dx \le CZ^{\frac{5}{3}}
\label{2-4}
\end{equation}
for the ground state energy. In follows from 
\begin{equation}
\int \rho_\Psi ^{\frac{5}{3}}(x)\,dx \le CZ^{\frac{7}{3}},
\label{2-5}
\end{equation}
equality $\int \rho_\Psi \,dx =N$ and assumption $N\lesssim Z$. To prove (\ref{2-5}) we
apply classical arguments of Lieb--Thirring, but replacing the Lieb--Thirring inequality by some relativistic inequalities (see Appendix \ref{sect-A}). Namely, let $\mathsf{b} \coloneqq  T - K U$ with 
$U=\rho_\Psi^{\frac{2}{3}} \varphi_{<} +
\beta^{-1}\rho_\Psi^{\frac{1}{3}} \varphi_{>}$ where $\varphi_{\gtrless}$ is a characteristic function of 
$\{x\colon \rho_\Psi \gtrless \beta^{-3}\}$.

 Consider multiparticle operator $B=\sum \mathsf{b}_{x_j}$ and its lowest eigenvalue $E_0$. Obviously, 
\begin{equation}
E_0 \le \blangle B\Psi,\Psi\brangle = \sum _j \blangle T_{x_j} \Psi, \Psi\brangle - 
K\int (\rho_\Psi^{\frac{5}{3}}\varphi_{<}  + \beta^{-1}\rho_\Psi^{\frac{4}{3}}\varphi_{>})\,dx.
\label{2-6}
\end{equation}
On the other hand, $E_0$ does not exceed the sum the sum of negative eigenvalues of   $\mathsf{b}$, and due to Daubechies inequality (\ref{A-1}) the absolute value of this sum does not exceed
\begin{equation}
C_0\int \max( U^{\frac{5}{2}}, \beta ^3 U^4)\,dx \le C_0 K^{\frac{5}{2}} 
\int (\rho_\Psi ^{\frac{5}{3}}\varphi_{<} + \beta^{-1}\rho_\Psi ^{\frac{4}{3}}\varphi_{>})\,dx.
\label{2-7}
\end{equation}
Therefore, assuming that $E_0\le 0$ we conclude that
\begin{equation*}
\sum _j \blangle T_{x_j} \Psi, \Psi\brangle - 
K\int \min(\rho_\Psi^{\frac{5}{3}},\, \beta^{-1}\rho_\Psi^{\frac{4}{3}}) + 
C_0 K^{\frac{5}{2}} 
\int (\rho_\Psi ^{\frac{5}{3}}\varphi_{<} + \beta^{-1}\rho_\Psi ^{\frac{4}{3}}\varphi_{>})\,dx \ge 0
\end{equation*}
and therefore for small positive constant $K$ we conclude  that 
\begin{equation}
\sum _j \blangle T_{x_j} \Psi, \Psi\brangle \ge  2\epsilon_0 \int (\rho_\Psi ^{\frac{5}{3}}\varphi_{<} + \beta^{-1}\rho_\Psi ^{\frac{4}{3}}\varphi_{>})\,dx.
\label{2-8}
\end{equation}

Thus, we proved that for any $\Psi\in \fH$ (\ref{2-8}) holds. Then 
\begin{multline}
 \sum_j \blangle H_{x_j}\Psi,\Psi\brangle  =\sum_j \blangle T_{x_j}\Psi,\Psi\brangle -\int V(x)\rho_\Psi(x)\, dx\ge\\
  \int (2\epsilon_0\rho_\Psi ^{\frac{5}{3}}- V(x)\rho_\Psi) \varphi_{<} \,dx
 + \int (2\epsilon_0 \beta^{-1}\rho_\Psi ^{\frac{4}{3}}- V(x)\rho_\Psi)\varphi_{>} \,dx.
 \label{2-9}
 \end{multline}
 We know, that this must be less than $-c_0 Z^{\frac{7}{3}}$ (it will follow, f.e. from the estimate from above). Observe that for $\ell(x) \ge a Z^{-\frac{1}{3}}$ we have $V(x)< a^{-1}Z^{\frac{4}{3}}$ and integral over this zone from $-V\rho_\Psi$ is greater than $ -C_0a^{-1} Z^{\frac{4}{3}}N$. Let us fix $a$ a large enough constant.
 
 Next, 
 \begin{gather*}
 \int _{x\colon \ell(x)\le a Z^{-1/3}} (\epsilon_0\rho_\Psi ^{\frac{5}{3}}- V(x)\rho_\Psi) \varphi_{<} \,dx \ge
 -C \int _{x\colon \ell(x)\le a Z^{-1/3}} V^{\frac{5}{2}}\,dx \ge - C_1 Z^{\frac{7}{3}}
 \end{gather*}
 and $(\epsilon_0 \beta^{-1}\rho_\Psi^{1/3}-V)\varphi_{>}$ is positive unless $\rho_\psi>\beta^{-3}$ and  
 $V\ge \epsilon_1 \beta^{-1}\rho_\Psi^{1/3}\ge \epsilon _1\beta^{-2}$ (and then $\ell(x)\le C_0\beta$. 
 
 Therefore we estimate $\int (\rho_\Psi^{5/3}\varphi_{<}\,+\beta^{-1}\rho_\Psi^{4/3}\varphi_{<} )\,dx$ from above by $CZ^{7/3}$ plus $\int_{x:\ell(x)\le C\beta} V\rho_\Psi \,dx$ and to get (\ref{2-4}) it is sufficient  to estimate this term. Further, it is sufficient to replace $V$ by $V_m$ (since $V=V_m + O(\beta^2)$ provided distance between nuclei is $\ge C\beta$). Also we can replace $V_m$ by $V_m+C\beta^{-2}$.
 
If $Z_m\beta \le \frac{2}{\pi}-\epsilon$, then we can decompose $H= \eta (H-V^1) +  (1-\eta )(H-V^0)$ where $(1-\eta )V^0$ coincides with $V$ in $\beta$-vicinity of $\y_m$ and equals $0$ outside of $2\beta$-vicinity of it 
and $V^1=\eta^{-1}(V-(1-\eta) V^0)$ and apply all above arguments for operator with $V=V^1$ while simply observing that $H-V^0$ is positive operator for $\eta $ sufficiently small. So we have proven that 
 
\begin{proposition}\label{prop-2-2}
Under assumption \textup{(\ref{1-10})} for the ground state 
\begin{equation}
\int \min (\beta^{-1}\rho _\Psi^{\frac{4}{3}}, \rho_\Psi^{\frac{5}{3}})\,dx \le CZ^{\frac{7}{3}}
\label{2-10}
\end{equation}
and \textup{(\ref{2-4})} holds.
\end{proposition}

Then we immediately arrive to Statement~\ref{cor-2-3-i} below, and Statement \ref{cor-2-3-ii} follows from \cite{Bach} and \cite{GS}:

\begin{corollary}\label{cor-2-3}
Under assumption \textup{(\ref{1-10})}
\begin{enumerate}[label=(\roman*), wide, labelindent=0pt]
\item\label{cor-2-3-i}
The following estimate hold:
\begin{equation}
\E_N \ge  
\Tr (H^-_{W+\lambda} ) -\frac{1}{2}\D (\rho,\rho)-CZ^{\frac{5}{3}} +\frac{1}{2}\D (\rho-\rho_\Psi,\, \rho-\rho_\Psi) 
\label{2-11}
\end{equation}
where $\rho,\lambda$ are arbitrary and $W=V-|x|^{-1}*\rho$.
\item\label{cor-2-3-ii}
Further,
\begin{multline}
\E_N \ge 
\Tr (H^-_{W+\lambda} ) -\frac{1}{2}\D (\rho,\rho)-\\
\frac{1}{2}\int |x-y|^{-1}\tr \bigl(e^\dag_N (x,y) e_N (x,y)\bigr) \,dxdy - CZ^{\frac{5}{3}-\delta}+
 \frac{1}{2}\D (\rho-\rho_\Psi,\, \rho-\rho_\Psi) 
\label{2-12}
\end{multline}
where $e_N(x,y)$ is the Schwartz kernel of the projector to $N$ lower eigestates of $H_W$.
\end{enumerate}
\end{corollary}

To cover\footnote{\label{foot-1} Unfortunately, only partially.} the critical case\footnote{\label{foot-2} I.e. with the non-strict inequality (\ref{1-9}) instead of (\ref{1-10}).} we will use (2.21) from \cite{SSS}
\begin{equation}
  \sum_{1\le i<j\le N} |x_i -x_j|^{-1}
  \ge \sum_{j=1}^N (\rho * |x|^{-1} * \Phi_s)(x_j) - \frac{1}{2}\D(\rho,\rho) - C N \varepsilon^{-1} ,
  \label{2-13}%{inequ:correlation}
\end{equation}
where again $\rho\ge 0$ is arbitrary $\lambda$ is arbitrary,  $\Phi \ge 0$ is spherically symmetric with $\int \Phi\,dx=1$, $\Phi_\varepsilon (x)=\varepsilon^{-3}\Phi (x/\varepsilon)$. Here $\frac{1}{2}$ is due to the difference in notations and
also now
\begin{equation}
W\coloneqq W_\varepsilon =V-|x| ^{-1}*\rho *\Phi_\varepsilon
\label{2-14}
\end{equation}
instead of (\ref{2-2}) and $-CNs^{-1}$ instead of the last term in (\ref{2-3}):

\begin{proposition}\label{prop-2-4}
Under assumption \textup{(\ref{1-9})}
\begin{equation}
\E_N \ge \Tr (H^-_{W+\lambda})+\lambda N -
\frac{1}{2}\D( \rho , \rho )-CN\varepsilon ^{-1}.
\label{2-15}
\end{equation}
\end{proposition}

\begin{remark}\label{rem-2-5}
\begin{enumerate}[label=(\roman*), wide, labelindent=0pt]
\item\label{rem-2-5-i}
Later we set $\varepsilon = Z^{-\frac{2}{3}}$. This would lead to $O(Z^{\frac{5}{3}})$ remainder estimate.
\item\label{rem-2-5-ii}
Proposition \ref{prop-2-4} falls short in two instances: there is no improved version of corollary \ref{cor-2-3}(ii) and also there is no ``bonus term'' $\frac{1}{2}\D( \rho -\rho_\Psi, \rho -\rho_\Psi )$ in the right-hand expression.
\end{enumerate}
\end{remark}

\section{Estimate from above}
\label{sect-2-2}
Estimate from above is straight-forward: we simply take $\Psi$ as a Slater determinant of $N$ lower eigenfunctions of $H_W$. If there are only  $N'<N$ negative eigenvalues then we take only $N'$ such eigenvalues, because $\E_N\le \E_{N'}$. Then we arrive to

\begin{proposition}
\begin{multline}
\E_N\le \Tr (H^-_{W+\lambda}) -\frac{1}{2}\D(\rho,\rho)  +\\
|\lambda-\nu|\cdot |\N^-_{W+\nu}-N| +
\D(\tr e_N(x,x)-\rho,\, \tr e_N(x,x,\nu)-\rho) -\\ 
\frac{1}{2}\int |x-y|^{-1}\tr \bigl(e^\dag_N (x,y) e_N (x,y)\bigr) \,dxdy
\label{2-16}
\end{multline}
with arbitrary $\rho$ and $\nu\le 0$, $W=V-|x|^{-1}*\rho$.
\end{proposition}

\chapter{Semiclassical methods}
\label{sect-3}

We will need the following semiclassical expressions: 
\begin{align}
P'(w) &= (2\pi )^{-3}q\int _{\{\xi: b (\xi)\le w\}}d \xi,
\label{3-1}\\
\intertext{and its integral}
P(w) &= (2\pi )^{-3}q\int _{\{\xi:b(\xi)\le w\}}b(\xi)\, d\xi, 
\label{3-2}
\end{align}
where in the non-relativistic case $b(\xi)= \frac{\hbar^2}{2\mu}|\xi|^2$ and respectively for $\mu=\hbar=1$
\begin{align}
&P^{\TF\,\prime}(w)=\frac{q}{6\pi^2} w_+ ^{\frac{3}{2}},
\label{3-3}\\
&P^{\TF}(w)=\frac{q}{15\pi^2} w_+ ^{\frac{5}{2}},
\label{3-4}
\end{align}
and in the relativistic case we have $b(\xi)= (c^2\hbar^2|\xi|^2 +\mu^2 c^4)^{\frac{1}{2}}-\mu c^2$ 
and respectively for $\mu=\hbar=1$
\begin{align}
&P^{\RTF\,\prime}(w)=\frac{q}{6\pi^2} w_+ ^{\frac{3}{2}}(1+\beta^2 w_+)^{\frac{3}{2}}
\label{3-5}
\end{align}
in the relativistic case $P^{\RTF} (w)$ is an \href{https://www.wolframalpha.com/input/?i=integrate+w+\%5E\%7B\%5Cfrac\%7B3\%7D\%7B2\%7D\%7D(1\%2Bbeta+w)\%5E\%7B\%5Cfrac\%7B3\%7D\%7B2\%7D\%7D+dw}{elementary function} as well and a sadistic Calculus instructor can give it on the test. 
However it turns out that we really do not need any separate relativistic Thomas-Fermi theory.
                
After scalings we have a \emph{semiclassical zone\/} $\cX_\scl\coloneqq\{x\colon \ell (x)\ge c Z^{-1}\}$, where the effective semiclassical parameter $h= 1/\zeta \ell$ and the operator is very similar to the non-relativistic one. There is also a \emph{singular zone\/} $\cX_\sing\coloneqq\{x\colon \ell (x)\le c Z^{-1}\}$ and it covers the \emph{relativistic zone\/} $\cX_\rel\coloneqq \{x\colon \ell (x)\le c \beta\}$.

Important is that
\begin{gather}
 0\le V(x)-W(x)\le C\zeta^2\coloneqq \min(Z^{\frac{4}{3}},Z\ell^{-1}),
  \label{3-6}\\
 |\partial ^\gamma (W-V)|\le C\zeta^2\ell(x)^{-|\gamma|}\qquad
 \forall \gamma:|\gamma|\le 2.
\label{3-7}
\end{gather}

\section{Trace term}
\label{sect-3-1}

Now the rescaling  methods  of \cite{monsterbook} allow us to prove the following:

\begin{proposition}\label{prop-3-1}
Let condition \textup{(\ref{1-9})} be fulfilled and let $W$ satisfy \textup{(\ref{3-6})} and \textup{(\ref{3-7})}.
\begin{enumerate}[label=(\roman*), wide, labelindent=0pt]

\item\label{prop-3-1-i}
Let $\psi_0 (x)$ be $\ell$-admissible function, equal $1$ in $\{x\colon \ell(x)\ge 2a\}$ and supported in $\{x\colon \ell(x)\ge a\}$. Then for $W=W^\TF$
\begin{multline}
|\Tr (H^-_{W+\lambda}\psi _0)- \int P^\RTF (W+\lambda)\psi_0 (x)\,dx|\le\\
C\left\{\begin{aligned}
&Z^{\frac{3}{2}} a^{-\frac{1}{2}} && a\le Z^{-\frac{1}{3}},\\
&Z^{\frac{5}{3}}(aZ^{\frac{1}{3}})^{-\delta} && a\ge Z^{-\frac{1}{3}}.
\end{aligned}\right.
\label{3-8}
\end{multline}

\item\label{prop-3-1-ii}
Let $\psi_m (x)$ be $\ell$-admissible, equal $1$ in
$\{x\colon |x-\y_m|\le a\}$ and supported in $\{x\colon |x-\y_m|\le 2a\}$. Then
 for $W=V_m=Z_m|x-\y_m|^{-1}$
\begin{equation}
|\int \bigl(\tr (e^1(x,x,0))-P^\RTF (V_m)\bigr)(1-\psi_m (x))\,dx |\le
Z^{\frac{3}{2}} d^{-\frac{1}{2}} 
\label{3-9}
\end{equation}
where $e^1(.,.,\tau)=\int_{-\infty}^\tau e(x,x,\tau')\,d\tau'$.
\end{enumerate}
\end{proposition}

\begin{proof}
Indeed, the contribution of the $\ell$-element of the partition to the remainder is $O(\zeta^3 \ell )$ exactly as in the non-relativistic case. Summation by partition elements results in the right-hand expression.
\end{proof}

Next, we need to consider vicinities of the singularities. Then the methods of Chapter~\ref{book_new-sect-25} of \cite{monsterbook} allow us to prove the following:

\begin{proposition}\label{prop-3-2}
In the framework of Proposition~\ref{prop-3-1} let $\phi_m$ be equal $1$ in $\{x\colon |x-\y_m|\le Z_m^{-1}\}$
and supported in $\{x\colon |x-\y_m|\le 2Z_m^{-1}\}$. Let $|\lambda|\le C_0Z d^{-1}$. Then
\begin{multline}
|\Tr (H^-_{W+\lambda}\psi _m (1-\phi_m) ) - \Tr (H^-_{V_m}\psi_m (1-\phi_m) ) + \\
\int \bigl(P^\RTF (W+\lambda)-P^\RTF (V_m)\bigr)\psi _m(x)(1-\phi_m(x))\,dx |\le\\
C\left\{\begin{aligned}
&Z^{\frac{3}{2}} d^{-\frac{1}{2}} && d\le Z^{-\frac{1}{3}},\\
&Z^{\frac{5}{3}} && d\ge Z^{-\frac{1}{3}}
\end{aligned}\right.
\label{3-10}
\end{multline}
where $d\ge cZ^{-1}$ is the minimal distance between nuclei.
\end{proposition}

\begin{proof}
Indeed, exactly as in the non-relativistic case, using methods of Sections~\ref{book_new-sect-12-5} and~\ref{book_new-sect-25-4} of \cite{monsterbook} we estimate  the contribution of $\ell$-element to the remainder by $O(\zeta \ell^3\bar{\zeta}^2\bar{\ell}^{-2})$  provided 
$Z^{-1+\delta}\lesssim \ell \lesssim d$ and by $O(\zeta^2 \ell ^2 \bar{\zeta}^2)$ provided 
$Z^{-1}\lesssim \ell \lesssim Z^{-1+\delta}$. This proves the required remainder estimate. For $d\le Z^{-1+\delta}$ we use a rescaling. 

Summation by partition elements results in the right-hand exptression.
\end{proof} 

\begin{remark}\label{rem-3-3}
We need to put cut-off $(1-\phi_m(x))$ because not only integrals of $P^\RTF (W+\lambda)$ and $P^\RTF (V_m)$ (of magnitude $\beta^3 Z^4\ell^{-4}$) and $P^{\RTF\,\prime} (W+\lambda)$ are diverging at $\y_m$, but even integral of their difference is logarithmically diverging.
\end{remark} 

Now we need to consider $CZ^{-1}$ vicinities of $\y_m$ and we will use the following Proposition:

\begin{proposition}\label{prop-3-4}
In the framework of Proposition~\ref{prop-3-1} 
\begin{enumerate}[label=(\roman*), wide, labelindent=0pt]
\item\label{prop-3-4-i}
$H_{W}\ge -C_0Z^2$.

\item\label{prop-3-4-ii}
Further
\begin{equation}
e(x,x,\lambda) \le CZ^{1-\delta}\ell(x)^{\delta-2}\qquad\text{for\ \ }|\lambda |\le c_0Z^2.
\label{3-11}
\end{equation}
\end{enumerate}
\end{proposition}

\begin{proof}
\begin{enumerate}[label=(\alph*), wide, labelindent=0pt]
\item\label{pf-3-4-a}
Assume first that $Z\asymp \beta^{-1}$ (i. e. $Z\ge \epsilon _0 \beta^{-1}$); then Statement~\ref{prop-3-4-i} follows immediately from Lieb-Yau inequality (Theorem~\ref{thm-A-2}): in the operator sense 
$H\ge \beta^{-1}\sqrt{\Delta} - \beta^{-2} -Z_m r^{-1}\ge -\beta^{-2}$,  $r=|x-\y_m|$.

Then $e(x,x,\lambda) \le C\ell(x)^{-3}h^{-3}$ with the semiclassical parameter $h$, which is $\asymp 1$ for $\ell \lesssim Z^{-1}$, $\lambda \lesssim Z^2$. Then
\begin{equation}
e(x,x,\lambda) \le C\ell(x)^{-3}\qquad\text{for\ \ }\lambda \le C_0Z^2,\ \ell(x)\lesssim Z^{-1}.
\label{3-12}
\end{equation}

Unfortunately it falls short for our needs. Let us shift $\y_m\mapsto 0$, and scale $x\mapsto Zx$, 
$\tau\mapsto Z^{-2}\tau$. Then we arrive to operator which modulo $O(1)$ is $\sqrt{\Delta} - Zr^{-1}$. Due to
\begin{equation}
 \sqrt{\Delta}-\frac{2}{\pi |x |}\ge A_s(\Delta)^{s}-B_s
\label{3-13}
\end{equation}
for any $s\in[0,1/2)$ and $A_s, B_s>0$ we can ``trade'' (due to Sobolev embedding theorem) $\ell^{-1+\delta}$ by $1$ in the scaled inequality (\ref{3-12}) and by $Z^{1-\delta}$ in the original one, thus arriving to (\ref{3-11}).

\item\label{pf-3-4-b}
Let us consider $Z\le \epsilon_0 \beta^{-1}$. Observe that in the operator sense
\begin{equation*}
H\ge (\frac{1}{4}\beta^{-2} r^{-2} + \beta^{-4})^{1/2}- Zr^{-1} -C\beta^{-2}\ge CZ^{-2};
\end{equation*} 
the latter inequality is proven separately for $r\lesssim \beta$ and for $r\gtrsim \beta$.

Moreover, we get $H\ge \epsilon_1\min (r^{-2}, \beta^{-1}r^{-1})$ for $r\le \epsilon_1 Z^{-1}$ and then we can trade  $\ell^{-3}$ to $CZ^3$ arriving even to the stronger version of (\ref{3-12}): namely, 
\begin{equation}
e(x,x,\lambda)\le CZ^3.
\label{3-14}
\end{equation}

Actually (\ref{3-14}) holds as $Z_m\beta \le 2\pi^{-1}-\sigma$ and could by quantified even for a parameter, rather than constant $\sigma>0$.
\end{enumerate}
\end{proof}

Then we immediately conclude that

\begin{corollary}\label{cor-3-5}
In the framework of Proposition~\ref{prop-3-1} for $|\lambda|\le C_0Zd^{-1}$ 
\begin{equation}
|\Tr (H_{W+\lambda}^-\phi_m )-\Tr (H_{V_m}^-\phi_m ) |\le CZd^{-1}.
\label{3-15}
\end{equation}
\end{corollary}

Now we can assemble all these results. However before doing this we replace $P^\RTF$ by $P^\TF$:

\begin{proposition}\label{prop-3-6}
\begin{enumerate}[label=(\roman*), wide, labelindent=0pt]
\item\label{prop-3-6-i}
Estimates \textup{(\ref{3-8})}, \textup{(\ref{3-9})} and \textup{(\ref{3-10})} hold with $P^\RTF$ replaced by $P^\TF$.

\item\label{prop-3-6-ii}
Estimate \textup{(\ref{3-10})} with $P^\RTF$ replaced by $P^\TF$ also holds with $\phi_m=0$.
\end{enumerate}
\end{proposition}

\begin{proof}
Statement~\ref{prop-3-6-i}  follows immediately from 
\begin{phantomequation}\label{3-16}\end{phantomequation}
\begin{multline}
P^{\RTF}(w)-P^{\TF}(w) \asymp  \beta ^2 w^{\frac{7}{2}},\qquad 
P^{\RTF\,\prime}(w)-P^{\TF\,\prime}(w) \asymp  \beta ^2w^{\frac{5}{2}}\\
 \text{for\ \ }  \beta^2 w\lesssim 1
\tag*{$\textup{(\ref{3-16})}_{1,2}$}\label{3-16-*}
\end{multline}
due to (\ref{3-5}). Statement~\ref{prop-3-6-ii}  follows immediately from
$P^{\TF}(w) \asymp  w^{\frac{5}{2}}$, $P^{\TF\,\prime}(w) \asymp  w^{\frac{3}{2}}$.
\end{proof}

\begin{remark}\label{rem-3-7}
Meanwhile, 
\begin{equation}
 \int \bigl(P^\RTF (V+\lambda)-P^\TF (V+\lambda)\bigr)\psi (x)\,dx \asymp \beta^2Z^4
\label{3-17}
\end{equation}
which could be as large as $Z^2$.
\end{remark}

Due to the scaling properties of $e(x,x,0)$ for $H=H_V$ and $P^\TF (V)$ for $V=V_m$  we conclude that
\begin{equation}
\int \bigl(\tr (e^1(x,x,0))- P^\RTF (V_m)\bigr)\,dx  = qZ_m^2 S(Z_m\beta)
\label{3-18}
\end{equation}
with unknown function $S(Z_m\beta)$. Indeed, if $\y_m=0$ then $x\mapsto x/k$ transforms operator with parameters 
$Z_m,\beta$ into operator with parameters $Z_m k$, $\beta k^{-1}$ multiplied by $k^{-2}$. 

\begin{remark}\label{rem-3-8}
Obviously, $S(Z_m\beta)$ monotone decreases as $\beta\to 0+$ and tends to $S(0)$ for the Schr\"odinger operator.
\end{remark}

Then due to (\ref{3-9}) for $V=V_m$
\begin{equation}
|\int \bigl(\tr (e^1(x,x,0)) - P^\TF (V_m)\bigr)\psi_m (x)\,dx -qZ_m^2 S(Z_m\beta) |\le
Z^{\frac{3}{2}} d^{-\frac{1}{2}}
\label{3-19}
\end{equation}
and we arrive to

\begin{proposition}\label{prop-3-9}
Let \textup{(\ref{1-9})} be fulfilled. Then  for $W=W^\TF$
\begin{multline}
|\Tr (H^-_{W+\lambda}) +  
\int P^\TF (W+\lambda)\,dx   - \sum_{1\le m\le M} qZ_m^2S(Z_m\beta)|\le\\
C\left\{\begin{aligned}
&Z^{\frac{3}{2}} d^{-\frac{1}{2}} && d\le Z^{-\frac{1}{3}},\\
&Z^{\frac{5}{3}} && d\ge Z^{-\frac{1}{3}}.
\end{aligned}\right.
\label{3-20}
\end{multline}
\end{proposition}

\section{Trace term. II}
\label{sect-3-2}

Let improve the above results for $d\gg Z^{-\frac{1}{3}}$. Observe first that the in this case the error in (\ref{3-8}) can be made $O(Z^{\frac{5}{3}}(d Z^{\frac{1}{3}})^{-\delta}+ Z^{\frac{5}{3}-\delta})$ provided we include  the relativistic Schwinger correction term. Since this term has a magnitude $Z^{\frac{5}{3}}$ and the contributions of the zones  $\{x \colon \ell (x) \le Z^{-\frac{1}{3}-\delta_1}\}$ and
 $\{x \colon \ell (x) \ge Z^{-\frac{1}{3}+\delta_1}\}$ in this term are $O(Z^{\frac{5}{3}-\delta})$, the  difference between relativistic and the standard non-relativistic Schwinger terms is 
$O(Z^{\frac{5}{3}-\delta})$ and we can use the latter
\begin{equation}
\mathsf{Schwinger}= (36\pi ) ^{\frac{2}{3}} q ^{\frac{2}{3}}
\int (\rho ^\TF )^{\frac{4}{3}}\,dx.
\label{3-21}
\end{equation}

Next, consider relativistic correction term 
\begin{equation}
\int \bigl(-P^\RTF (W+\lambda) + P^\RTF (V_m) + P^\TF (W+\lambda) - P^\TF (V_m)\bigr)\psi_m (1-\phi_m)\,dx.
\label{3-22}
\end{equation}

Again, one can see easily that the contributions of these two zones  $\{x \colon \ell (x) \le Z^{-\frac{1}{3}-\delta_1}\}$ and 
$\{x \colon \ell (x) \ge Z^{-\frac{1}{3}+\delta_1}\}$ in this term are $O(Z^{\frac{5}{3}-\delta})$, so we need to consider the contribution of the zone ${\{x \colon Z^{-\frac{1}{3}-\delta_1} \le \ell (x) \le Z^{-\frac{1}{3}+\delta_1}\}}$, where due to (\ref{3-5})
$P^\RTF(w)-P^\TF (w) = \frac{q}{14\pi^2} \beta^2 w_+^{\frac{7}{2}} + O(Z^{\frac{8}{3}-\delta})$ (for both $w=W^\TF+\lambda$ and $w=V_m$) and therefore modulo  the same error expression (\ref{3-22}) coincides with 
\begin{equation}
\mathsf{RCT}\coloneqq   
\frac{q}{14\pi^2} \beta^2 \int \bigl( -(W^\TF+\lambda) _+^{\frac{7}{2}}+ V ^{\frac{7}{2}}\bigr)\,dx
\label{3-23}
\end{equation}
with the integral taken over this zone or $\bR^3$ (does not matter). Then we arrive to

\begin{proposition}\label{prop-3-10}
Let \textup{(\ref{1-9})} be fulfilled and $d\ge Z^{-\frac{1}{3}}$. Then  for $W=W^\TF$
\begin{multline}
|\Tr (H^-_{W+\lambda}) +  
\int P^\TF (W+\lambda)\,dx   - \sum_{1\le m\le M} qZ_m^2S(Z_m\beta) - \\ \mathsf{Schwinger}-  \mathsf{RCT} |\le 
C\bigl(Z^{\frac{5}{3}}(d Z^{\frac{1}{3}})^{-\delta}+ Z^{\frac{5}{3}-\delta}\bigr).
\label{3-24}
\end{multline}
\end{proposition}

\section{Trace term. III}
\label{sect-3-3}

Obviously, all these results hold for $W=W_\varepsilon$ defined by (\ref{2-14}) with $\rho=\rho^\TF$. However we need to
estimate an error when we replace $W_\varepsilon$ by $W^\TF$. One can prove easily that
\begin{equation}
|W_\varepsilon-W^\TF |\le C _s (Z\ell ^{-1})^{\frac{3}{2}} \varepsilon ^2 (\varepsilon \ell{-1})^s
\label{3-25}
\end{equation}
with arbitrary $s$ for $\ell \le \epsilon_0 Z^{-\frac{1}{3}}$ and with $s=\frac{1}{2}$ $\ell \le \epsilon_0 Z^{-\frac{1}{3}}$ and therefore 
\begin{equation}
|\int \bigl(P^\TF (W_\varepsilon+\lambda) - P^\TF (W^\TF+\lambda)|\,dx|\le
CZ^3\varepsilon^2;
\label{3-26}
\end{equation}
adding error $CZ\varepsilon^{-1}$ in (\ref{2-15}) we get $C(Z^3 \varepsilon^2+Z\varepsilon^{-1})$. It reaches minimum $CZ^{\frac{5}{3}}$ as $\varepsilon \asymp Z^{-\frac{2}{3}}$ and we arrive to

\begin{proposition}\label{prop-3-11}
Let \textup{(\ref{1-9})} be fulfilled. Then  for $W=W_\varepsilon$ with $\varepsilon=Z^{-\frac{2}{3}}$ 
\textup{(\ref{3-20})} holds and the left-hand expression of \textup{(\ref{3-26})} is $O(Z^{\frac{5}{3}})$.
\end{proposition}

\section{$\N$- and $\D$-terms}
\label{sect-3-4}

For these terms (needed for the estimate from above) arguments are simpler; let $\phi_0=1-\phi_1-\ldots-\phi_M$. 

\begin{proposition}\label{prop-3-12}
In the framework of Proposition~\ref{prop-3-1} 
\begin{enumerate}[label=(\roman*), wide, labelindent=0pt]
\item\label{prop-3-12-i}
The following estimates hold
\begin{gather}
|\int \bigl (e (x,x,\lambda ) - P^{\RTF\,\prime} (W+\lambda)\bigr) \phi_0(x)\,dx|
\le C Z^{\frac{2}{3}}
\label{3-27}\\
\intertext{and for $d\ge Z^{-\frac{1}{3}}$}
|\int \bigl (e (x,x,\lambda ) - P^{\RTF\,\prime} (W+\lambda)\bigr) \phi_0(x)\,dx|
\le 
C \bigl(Z^{\frac{2}{3}}(dZ^{\frac{1}{3}})^{-\delta} + Z^{\frac{2}{3}-\delta}\bigr).
\label{3-28}
\end{gather}
\item\label{prop-3-12-ii}
Further,
\begin{equation}
|\int e (x,x,\lambda ) \phi _m(x)\,dx| \le C .
\label{3-29}
\end{equation}
\item\label{prop-3-12-iii}
Finally,
\begin{equation}
|\int \bigl ( P^{\RTF\,\prime} (W+\lambda)- P^{\TF\,\prime} (W+\lambda)\bigr) \phi_0 (x)\,dx|
\le \\
C Z^{\frac{1}{3}}.
\label{3-30}
\end{equation}
\end{enumerate}
\end{proposition}

\begin{proposition}\label{prop-3-13}
In the framework of Proposition~\ref{prop-3-1} 
\begin{enumerate}[label=(\roman*), wide, labelindent=0pt]
\item\label{prop-3-13-i}
The following estimates hold
\begin{multline}
\D \bigl( (e (x,x,\lambda ) - P^{\RTF\,\prime} (W+\lambda))\phi_0,\, 
(e (x,x,\lambda ) - P^{\RTF\,\prime} (W+\lambda))\phi_0\bigr)\\
\le C Z^{\frac{5}{3}}
\label{3-31}
\end{multline}
and for $d\ge Z^{-\frac{1}{3}}$
\begin{multline}
\D \bigl( (e (x,x,\lambda ) - P^{\RTF\,\prime} (W+\lambda))\phi_0,\, 
(e (x,x,\lambda ) - P^{\RTF\,\prime} (W+\lambda))\phi_0\bigr)\\
 \\
\le C Z^{\frac{5}{3}}(dZ^{\frac{1}{3}})^{-\delta} + CZ^{\frac{5}{3}-\delta}.
\label{3-32}
\end{multline}
\item\label{prop-3-13-ii}
Further,
\begin{equation}
\D\bigl( e (x,x,\lambda ) \phi _m(x)\,e (x,x,\lambda ) \phi _m(x)\bigr) \le C Z.
\label{3-33}
\end{equation}
\item\label{prop-3-13-iii}
Finally,
\begin{multline}
\D \bigl( ( P^{\RTF\,\prime} (W+\lambda)- P^{\TF\,\prime} (W+\lambda))\phi_0,\\ 
(P^{\RTF\,\prime} (W+\lambda)- P^{\TF\,\prime} (W+\lambda))\phi_0\bigr)\le
CZ.
\label{3-34}
\end{multline}
\end{enumerate}
\end{proposition}

\begin{proof}[Proof of Propositions~\ref{prop-3-12} and \ref{prop-3-13}]

Proof is straightforward:

Statements (i) are proven by the semiclassical scaling technique exactly as in \cite{monsterbook}, Chapter 25.

Statements (ii) follow from Proposition~\ref{prop-3-4}.
Statements (iii) follow from (\ref{3-5}) and properties $W^\TF$.
\end{proof}

\section{Dirac term}
\label{sect-3-5}

Finally, consider $-\frac{1}{2}\iint \tr \bigl(e^\dag _N(x,y) e _N(x,y) \bigr)\,dxdy$. The main contribution to it is delivered by the zone $\cY\times \cY$ with 
$\cY=\{ x \colon Z^{-\frac{1}{3}-\delta_1}\le \ell (x)\le Z^{-\frac{1}{3}+\delta_1}\}$ and in this zone non-magnetic approximation delivers correct the expression
\begin{equation}
\mathsf{Dirac}= -\frac{9}{2}(36\pi ) ^{\frac{2}{3}}
q ^{\frac{2}{3}} \int (\rho ^\TF )^\frac{4}{3}\,dx,
\label{3-35}
\end{equation}
with an error $Z^{\frac{5}{3}-\delta}$.

\chapter{Main theorems}
\label{sect-4}

Now repeating arguments of Section~\ref{book_new-sect-25-4} of \cite{monsterbook} we arrive to our main results:

\begin{theorem-foot}\footnotetext{\label{foot-3} Cf. Theorems~\ref{book_new-thm-25-4-7} and~\ref{book_new-thm-25-4-12} of \cite{monsterbook}.}\label{thm-4-1}
Let assumption \textup{(\ref{1-9})} be fulfilled. Then
\begin{enumerate}[label=(\roman*), wide, labelindent=0pt]
\item\label{thm-4-1-i}
The following asymptotic holds 
\begin{equation}
\E_N = \cE^\TF_N + \mathsf{Scott} +O\bigl(Z^{\frac{5}{3}} + Z^{\frac{3}{2}}d^{-\frac{1}{2}}\bigr).
\label{4-1}
\end{equation}
Recall that $\mathsf{Scott}=q\sum Z_m^2 S(Z_m\beta)$ and $d$ is the minimal distance between nuclei. 

\item\label{thm-4-1-ii}
Furthermore, let assumption \textup{(\ref{1-10})} be fulfilled. Then for $d\ge Z^{-\frac{1}{3}}$
\begin{multline}
\E_N = \cE^\TF_N + \mathsf{Scott} +\mathsf{Dirac}+\mathsf{Swinger} +\mathsf{RCT} + \\
O\bigl(Z^{\frac{5}{3}}(dZ^{\frac{1}{3}})^{-\delta}+ Z^{\frac{5}{3}-\delta}\bigr).
\label{4-2}
\end{multline}
\end{enumerate}
\end{theorem-foot}

\begin{remark}\label{rem-4-2}
\begin{enumerate}[label=(\roman*), wide, labelindent=0pt]

\item\label{rem-4-2-i}
For the improved upper estimate in (\ref{4-2}) we do not need assumption (\ref{1-10}). 

\item\label{rem-4-2-ii}
These theorems allow us to consider the free nuclei model and recover Theorem~\ref{book_new-thm-25-4-13} of \cite{monsterbook}, albeit without assumption (\ref{1-10}) we get only $\delta=0$.

\item\label{rem-4-2-iii}
We also recover estimate 
\begin{equation}
|\lambda_N-\nu|\le C\left\{\begin{aligned}
& Z^{\frac{8}{9}} &(Z-N)_+\le Z^{\frac{2}{3}},\\
& (Z-N)_+^{\frac{1}{3}} &(Z-N)_+\ge Z^{\frac{2}{3}},
\end{aligned}\right.
\end{equation}
where $\nu$ is a chemical potential and $\lambda_N$ is the $N$-th lowest eigenvalue of $H_{W^\TF}$ (reset to $0$ if there are less than $N$ negative eigenvalues). Furthermore, for $d\ge Z^{-\frac{1}{3}}$ one can include the factor $((dZ^{\frac{1}{3}})^{-\delta}+Z^{-\delta})$ intto the right-hand expression.
\end{enumerate}
\end{remark}

\begin{theorem-foot}\footnotetext{\label{foot-4} Cf. Theorem~\ref{book_new-thm-25-4-14} of \cite{monsterbook}.}
\label{thm-4-3}
Let assumption \textup{(\ref{1-10})} be fulfilled. Then
\begin{enumerate}[label=(\roman*), wide, labelindent=0pt]

\item\label{thm-4-3-i}
The following estimate holds:
\begin{equation}
\D(\rho_\Psi-\rho^\TF,\, \rho_\Psi-\rho^\TF)\le C Z^{\frac{5}{3}}.
\label{4-4}
\end{equation}

\item\label{thm-4-3-ii}
Furthermore,  for $d\ge Z^{-\frac{1}{3}}$
\begin{equation}
\D(\rho_\Psi-\rho^\TF,\, \rho_\Psi-\rho^\TF)\le C (Z^{\frac{5}{3}}(dZ^{\frac{1}{3}})^{-\delta}+ Z^{\frac{5}{3}-\delta}).
\label{4-5}
\end{equation}
\end{enumerate}
\end{theorem-foot}

\begin{remark}\label{rem-4-4}
\begin{enumerate}[label=(\roman*), wide, labelindent=0pt]
\item\label{rem-4-4-i}
Estimates (\ref{4-4}) and (\ref{4-5}) allow us to consider the excessive negative charge and ionization energy and, repeating arguments of Section~\ref{book_new-sect-25-5} of \cite{monsterbook},  to recover  Theorems~\ref{book_new-thm-25-5-2} and~\ref{book_new-thm-25-5-3}.
\item\label{rem-4-4-ii}
 Further, these estimates allow us to consider the excessive positive charge in the free nuclei model and, repeating arguments of Section~\ref{book_new-sect-25-6} of \cite{monsterbook},  to recover~\ref{book_new-thm-25-6-4}.
 \end{enumerate}
\end{remark}

\begin{remark}\label{rem-4-5}
We can even make a poor man  version of (\ref{4-2}) in the critical case, when only assumption (\ref{1-9}) is fulfilled. \begin{enumerate}[label=(\roman*), wide, labelindent=0pt]

\item\label{rem-4-5-i}
Consider how our terms depend on $q$. In the atomic  case consider given $Z$, $N$ and shift to $\y_1=0$. Then
\begin{equation}
\rho^\TF _q (x)=q^2\rho^\TF _1(q^{\frac{2}{3}}x),\qquad W^\TF _q (x)=q^{\frac{2}{3}}W^\TF_1(q^{\frac{2}{3}}x)
\label{4-6}
\end{equation}
and $\cE^\TF\asymp q^{\frac{2}{3}}Z^{\frac{7}{3}}$, $\mathsf{Scott}\asymp qZ^2$,
$\mathsf{Dirac}\asymp \mathsf{Schwinger}\asymp q^{\frac{4}{3}}Z^{\frac{5}{3}}$, while
$\mathsf{RCT}\asymp q^{\frac{4}{3}}\beta^2Z^{\frac{11}{3}}$.

\item\label{rem-4-5-ii}
Repeating corresponding arguments of \cite{SSS}, one can prove that in the correlation inequality (\ref{2-13}) the constant is $C(q)\le C_0q^{\frac{2}{3}}$. On the other hand, we use the estimate for 
$|W-W_\varepsilon| \asymp q \varepsilon^2 Z^{\frac{3}{2}}\ell^{-\frac{3}{2}}$ 
and then approximation error is $C_0 Z^3q^2 \varepsilon ^2$. Optimizing $Z^3q^2 \varepsilon ^2+Zq^{\frac{2}{3}}\varepsilon^{-1}$ by $\varepsilon$ we get $Cq^{\frac{10}{9}}Z^{\frac{5}{3}}$ and for large constant $q$ it is less than  $q^{\frac{4}{3}}$. In the ``real life'' $q=2$.
\end{enumerate}
\end{remark}

\begin{appendices}
\chapter{Appendix: Some  inequalities}
\label{sect-A}

We follow \cite{SSS} with some modifications:

The following two inequalities we recall are crucial in many of our
estimates. They serve as replacements for the Lieb-Thirring inequality
\cite{LT} used in the non-relativistic case.

\begin{theorem}[Daubechies inequality]\label{thm-A-1}
\begin{enumerate}[label=(\roman*), wide, labelindent=0pt]
\item\label{thm-A-1-i}
\underline{One-body case}: 
\begin{equation}
   \Tr\bigl[(\beta^{-2}\Delta +\beta^{-4})^{\frac{1}{2}}-\beta^{-2} - V(\x)\bigr]^{-} \ge 
   -C \int \Bigl( V_+^{(n+2)/2} + \beta^n  V_+^{n+1}\bigr)\,dx.
\label{A-1} 
\end{equation} 
where $n\ge 3$ is a dimension.

\item\label{thm-A-1-ii}\underline{Many-body case}:
Let  $\Psi\in\bigwedge_{j=1}^N \sL^2(\bR^3;\bC^q)$ and let $\rho_{\Psi}$ be its one-particle density. Then for $n=3$
\begin{gather}
\langle \sum_{j=1}^N
  \big[(\beta^{-2}\Delta _j +\beta^{-4})^{\frac{1}{2}}-\beta^{-2}\big]\Psi,\Psi \rangle
  \ge \int \min \bigl(\rho_\Psi ^{\frac{5}{3}}, \beta^{-1}\rho_\Psi^{\frac{4}{3}}\bigr)\,dx.
\label{A-2}
\end{gather}
\end{enumerate}
\end{theorem}

This theorem also holds in the non-relativistic limit $\beta= 0$ and operator 
$(\beta^{-2}\Delta +\beta^{-4})^{\frac{1}{2}}-\beta^{-2}$ replaced by 
$\frac{1}{2}\Delta$.

\begin{theorem}[Lieb-Yau inequality]\label{thm-A-2} 
Let $n=3$. Let $C>0$ and $R>0$ and let
\begin{equation}
  H_{C,R} = \Delta^{\frac{1}{2}} - \frac{2}{\pi|x|}  - C/R.
\label{A-3} 
\end{equation}
Then, for any density matrix $\gamma$ and any function $\theta$ with
support in $B_R = \{x\,|\,|x|\le R\}$  
\begin{equation}
  \Tr\bigl[\bar{\theta}\gamma \theta H_{C,R}\bigr] \ge -4.4827 \,C^4
  R^{-1} \{3/(4\pi R^3)  
  \int |\theta(x)|^2\, dx\}.
  \label{A-4} 
\end{equation}
\end{theorem}

Note that when $\theta=1$ on $B_R$ then the term inside the brackets
$\{ \}$ equals~$1$. 
 
\begin{theorem}[Critical Hydrogen inequality]\label{thm-A-3}
 Let $n=3$. For any $s\in[0,1/2)$ there exists constants $A_s, B_s>0$ such that
\begin{equation}
   \Delta^{\frac{1}{2}}-\frac{2}{\pi|x |}\ge  A_s\Delta ^{s}-B_s.
    \label{A-5}
\end{equation}
\end{theorem}

\begin{theorem}[Hardy-Littlewood-Sobolev inequality]\label{thm-A-4}
There exists a constant $C$ such that 
\begin{equation}
    \D(f)\coloneqq \iint |x-y|^{-1}f(x)f^\dag (y)\,dxdy \le C\,\| f \|_{\sL^{6/5}}^2.
     \label{A-6} 
\end{equation}
\end{theorem}

\end{appendices}

\end{document}